\documentclass[a4paper,11pt,DIV=12, headings=standardclasses, headings=small]{scrartcl}

\usepackage{nag}
\usepackage[utf8]{inputenc}
\usepackage[english]{babel}

\usepackage{amscd} 

\usepackage{amsmath}
\usepackage{amsthm}
\usepackage{amsfonts}
\usepackage{amssymb}
\usepackage{mathrsfs} 
\usepackage{arydshln} 
\usepackage{url} 
\usepackage{graphicx}
\usepackage[numbers]{natbib}
\usepackage{tikz}
\usepackage{tikz-cd}
\usetikzlibrary{arrows,arrows.meta, shapes, shapes.geometric, shapes.misc, trees, graphs, matrix,  automata, backgrounds, calendar, positioning,external}
\tikzcdset{arrow style=tikz, diagrams={>=stealth}}
\usepackage{pgfplots}   
\pgfplotsset{width=10cm,compat=1.9}

\usepackage{pxfonts}
\usepackage[T1]{fontenc} 

\usepackage{hyperref}
\hypersetup{pdfpagelabels,hidelinks}

\usepackage[textsize=tiny,background color=white]{todonotes}
\DeclareMathOperator{\SO}{SO}

\DeclareMathOperator{\Ker}{Ker}

\DeclareMathOperator{\Hom}{Hom}

\DeclareMathOperator{\Aut}{Aut}

\newcommand{\C}{\mathbb{C}}
\newcommand{\Z}{\mathbb{Z}}
\newcommand{\Q}{\mathbb{Q}}

\newcommand{\HH}{\mathbb{H}}

\newcommand{\gL}{\mathfrak{g}}

\newcommand{\cC}{{\mathcal C}}

\newcommand{\va}{\varphi}

\newcommand{\ra}{{\rightarrow}}

\newcommand{\hra}{\hookrightarrow}
\newcommand{\lra}{\longrightarrow}

\newcommand{\half}{{{\frac{1}{2}}}}
\newcommand{\sslash}{\mathbin{/\mkern-6mu/}}
\newcommand{\ssslash}{\mathbin{\big/\mkern-10mu\big/}}

\DeclareMathOperator{\B}{B}

\DeclareMathOperator{\haut}{hAut}

\DeclareMathOperator{\Map}{Map}

\DeclareMathOperator{\fw}{fw}

\DeclareMathOperator{\Der}{Der}

\DeclareMathOperator{\Diff}{Diff}

\DeclareMathOperator{\Kdim}{Kdim }

\DeclareMathOperator{\MC}{MC}

\numberwithin{equation}{section}
 
\makeatletter
\newtheorem*{rep@theorem}{\rep@title}
\newcommand{\newreptheorem}[2]{%
\newenvironment{rep#1}[1]{%
 \def\rep@title{#2 \ref{##1}}%
 \begin{rep@theorem}}%
 {\end{rep@theorem}}}
\makeatother

\newtheorem{thm}{Theorem}[section]
\newreptheorem{theorem}{Theorem}

\newtheorem{lem}[thm]{Lemma}
\newtheorem{prop}[thm]{Proposition}
\newtheorem{conj}[thm]{Conjecture}

\newtheorem{thmL}{Theorem}

\theoremstyle{definition}
\newtheorem{defn}[thm]{Definition}
\newtheorem{rem}[thm]{Remark}

\theoremstyle{remark}

\makeatletter
\renewenvironment{proof}[1][\proofname] {\par\pushQED{\qed}\normalfont\topsep6\p@\@plus6\p@\relax\trivlist\item[\hskip\labelsep\bfseries#1\@addpunct{.}]\ignorespaces}{\popQED\endtrivlist\@endpefalse}
\makeatother

\begin{document}
\title{\Large Tautological rings of fake quaternionic spaces}
\author{\large Nils Prigge}
\date{}

\maketitle

\begin{abstract}
	\noindent
	\textbf{Abstract.} 
	The tautological ring $R^*(M)$ of a smooth manifold $M$ is the ring of characteristic classes generated by the Miller-Morita-Mumford classes, and is often more accessible than the ring of all characteristic classes of smooth $M$-bundles.
	 
	In this paper, we show that the Krull dimension of the tautological ring vanishes for almost all manifolds homotopy equivalent to $\HH P^2$ through a combination of new methods in rational homotopy theory developed in \cite{Ber17, Ber20} and the family signature theorem. 
\end{abstract}

\section{Introduction}

Let $M^d$ be a closed, oriented, smooth manifold and let $\pi:E\ra B$ be a proper submersion of oriented manifolds with fibre $M$. We denote by $T_{\pi}E:=\Ker(D\pi:TE\ra TB)$ the vertical tangent bundle which is an oriented vector bundle of rank $d$. For any class $c\in H^{|c|}(\B\SO(d))$ we can define a characteristic class of $\pi$ by the following fibre integral
\begin{equation}\label{MMM}
\kappa_c(\pi):=\int_{\pi}c(T_{\pi}E)\in H^{|c|-d}(B),
\end{equation}
which is called a \emph{generalized Miller-Morita-Mumford class} or \emph{tautological class} \cite{GGRW17} and the subring of $H^*(B)$ generated by all MMM-classes is called the \emph{tautological ring} $R^*(\pi)$.\footnote{The terminology originates from the special case of surface bundles where it has been coined by Mumford \cite{Mum83} and is a classical object of study.} In particular, we can study the tautological ring of the universal oriented $M$-bundle over the classifying space $\B\Diff^+(M)$, denoted by $R^*(M)\subset H^*(\B\Diff^+(M))$, which provides information for tautological rings of arbitrary such bundles. 

\medskip
Results in the literature about ring theoretic properties of $R^*(M)$ (for high-dimensional manifolds) generally use one of the following three methods:
\begin{itemize}
	\item[(i)] One can produce relations in the tautological ring that only use the underlying fibration and don't depend on the bundle structure, for example by analysing the Serre spectral sequence with some constraints on the cohomology ring $H^*(M)$ \cite{Gri17,GGRW17}. A very streamlined account has been given by Randal-Williams, who showed (amongst other things) that $R^*(M)$ is finitely generated as a ring if $H^*(M)$ is concentrated in even degrees \cite[Thm A]{RW16}.
	\item[(ii)] A family version of Hirzebruch's signature theorem implies that $\kappa_{L_i}$ vanishes if $d$ is odd and $\kappa_{L_i}$ is nilpotent if $d$ is even, where $L_i\in H^{4i}(\B\SO(d))$ denotes the Hirzebruch L-classes (we recall a precise statement in Section \ref{signature} below).
	\item[(iii)] Given a torus action $\mathbb{T}\curvearrowright M$, one can efficiently compute the tautological classes of the associated $M$-bundle $\pi:M\sslash\mathbb{T}\ra *\sslash\mathbb{T}=\B \mathbb{T}$. In particular, one can potentially deduce that certain MMM-classes are non-trivial and more generally obtain lower bounds on the Krull dimension of $ R^*(M)$.
\end{itemize}

Recently, Alexander Berglund has described in \cite{Ber17,Ber20} a program how to do step (i) systematically using rational homotopy theory. In this paper, we use his results combined with the relations obtained by the family signature theorem to study the tautological ring of manifolds $M$ that are homotopy equivalent to $\HH P^2$. Such manifolds can be classified using surgery theory, and in fact there are infinitely many manifolds $M\simeq \HH P^2$ \cite{Hsi66, EKu62} which are classified by the Pontrjagin number $p_1^2(M)$ \cite{KrSt07}. 

Our main result (cf.\,Theorem \ref{Main}) shows that the Krull dimension of $R^*(M)$ almost always vanishes. This demonstrates that the Krull dimension of the tautological ring is almost a homotopy invariant for fake quaternionic spaces and suggests that some ring theoretic properties of the tautological ring seem to depend only mildly on the manifold structure which is quite surprising.

\begin{thmL}\label{MainTheoremIntro}
	For almost all manifolds $M\simeq \HH P^2$ we have $\Kdim R^*(M)=0$.
\end{thmL}

Berglund's approach is based on studying fibrations with extra structure such as a vector bundle on the total space whose restriction to a fibre is equivalent to a fixed vector bundle. We recall a suitable version of such fibrations in Section \ref{RationalReview} and discuss rational models of the corresponding classifying spaces for fake quaternionic spaces. We prove the main theorem in Section \ref{ParametrizedHirzebruchSection}. The new idea in this paper is to consider the tautological ring as parametrized by the value of $p_1(M)$ and then use a version of Grothendieck's generic freeness theorem to obtain information how the dimensions can vary. In Section \ref{endsection} we discuss possible exceptions to Theorem \ref{MainTheoremIntro} and highlight an interesting connection to the $\hat{A}$-genus.

\bigskip
\noindent
\textbf{Acknowledgements.}
This paper is extracted from the author's thesis and I would like to express my gratitude to Oscar Randal-Williams for his supervision. I also want to thank Alexander Berglund for many fruitful discussions and sharing of his work which is an important basis of this paper. This research was supported by the Knut and Alice Wallenberg foundation through grant no.\,2019.0519.

\section{Preliminaries}

\subsection{Algebraic conventions}

We recall some standard constructions and terminology of differential graded algebra and notation introduced in Berglund's papers \cite{Ber17,Ber20}. We use cohomological grading conventions for differential graded commutative algebras (cdga) and homological gradings for differential graded Lie algebras. We list a few standard constructions to fix our conventions.
\begin{itemize}
 \item[(i)] To any cdga $A$ we can associate a differential graded Lie algebra $(\Der(A),[d_A,-])$, where, due to our degree convention, a derivation $\theta\in \Der(A)$ of degree $n$ is a linear map $\theta:A^*\ra A^{*-n}$ that satisfies $\theta(a\cdot a')=\theta(a)\cdot a'+(-1)^{n\cdot |a|}a\cdot \theta(a')$ for all $a,a'\in A$.
 \item[(ii)] Given a cdga $A$ of finite type and a dg Lie algebra $L$, the completed tensor product $A\hat{\otimes} L$, which in degree $n$ is given by $\prod_iA^i\otimes L_{i+n}$, is a dg Lie algebra with the usual differential of tensor products and bracket given by 
 \[[a\otimes l,a'\otimes l]=(-1)^{|a'|\cdot |l|} a\cdot a'\otimes [l,l'].\]
 \item[(iii)] Given a dg Lie algebra $L$ and $n\in \Z_{\geq 0}$, we denote by $L\langle n\rangle\subset L$ the dg Lie subalgebra determined by 
 \begin{equation*}
  L\langle n\rangle_i=\begin{cases}
                       L_i & i>n\\
                       \ker(d_L:L_i\ra L_{i-1}) & i=n\\
                       0 & i<n
                      \end{cases}
 \end{equation*}
 In keeping with commonly used notation, we denote $\Der(\Lambda)\langle 1\rangle$ by $\Der^+(\Lambda)$.
\item[(iv)] A Maurer-Cartan element $\tau$ in a dg Lie algebra $L$ is an element that satisfies the Maurer-Cartan equation $d_L(\tau)+\half[\tau,\tau]=0$. The set of Maurer-Cartan elements is denoted by $\text{MC}(L)$ and for $\tau\in \MC(L)$ the twisted dg Lie algebra $L^{\tau}$ is the same as $L$ as a graded Lie algebra but with twisted differential $d_{L^{\tau}}=d_L+[\tau,-]$.
 \end{itemize}

We follow \cite{Ber20} in our convention for the \emph{Chevalley-Eilenberg complex} of a dg Lie algebra (which in turn is based on \cite{Ta83}), which is defined as the dg coalgebra \[\cC^{CE}(L)=(\Lambda sL,D=d_1+d_2),\] where the differentials are determined by their corestrictions
\begin{align*}
 d_1(sl)&=-sd_L(l), & d_2(sl_1\wedge sl_2)=(-1)^{|l_1|}s[l_1,l_2].
\end{align*}
If a dg Lie algebra acts on a cdga $A$ through derivations, then the Chevalley-Eilenberg cochain complex is
\[\cC_{CE}(L;A):=\big(\Hom(\cC^{CE}(L),A),\partial+t\big),\]
where 
\begin{equation}\label{CEconvention}
 \begin{split}
   \partial(f)&=d_A\circ f-(-1)^{|f|}f\circ D\\
 t(f)(sl_1\wedge \hdots \wedge sl_n)&=\sum_{i=1}^n(-1)^{|sl_i|(|f|+|sl_1|+\hdots+|sl_{i-1}|)}l_i\cdot f(sl_1\wedge \hdots \hat{sl_i} \hdots \wedge sl_n),
 \end{split}
\end{equation}
and $\cC_{CE}(L;A)$ is a cdga via the convolution product. An element $f\in \cC^{CE}(L;A)$ is an \emph{$n$-cochain} if $f(sl_1\wedge\hdots\wedge sl_k)=0$ unless $k=n$ and we identify   $0$-cochains with $A$.

\subsection{Orientations and fibre integration}
Throughout the text we consider rational (co)homology groups unless stated otherwise. We assume that all manifolds are oriented throughout and that all maps preserve the orientation and we denote the set of orientation preserving diffeomorphisms (resp.\,homotopy self-equivalences) by $\Diff^+(M)$ (resp.\,$\haut^+(M)$) and the path component of the identity by $\Diff_0(M)$ (resp.\,$\haut_0(M)$). 

\smallskip
A fibration $\pi:E\ra B$ whose fibre is a Poincar\'e duality space $X$ of formal dimension $d$ is called orientable if the local coefficient system $\mathcal{H}^d(X)$ on $B$ is trivial, and oriented if we choose an isomorphism $\mathcal{E}_X:\mathcal{H}^d(X;\Z)\overset{\cong}{\lra}\Z$. Given an orientation $\mathcal{E}_X$, one can define fibre integration via the Serre spectral sequence as the composition
\begin{equation}
 \int_{\pi}:H^*(E)\twoheadrightarrow E^{*-d,d}_{\infty}\subset E^{*-d,d}_2\cong H^{*-d}(B;\mathcal{H}^d(X))\xrightarrow{(\mathcal{E}_X)_*} H^{*-d}(B),
\end{equation}
where $H^*(E)$ projects onto $E_{\infty}^{*-d,d}$ as $H^*(X)$ vanishes for $*>d$. Occasionally, we also denote fibre integration by $\pi_!:H^*(E)\ra H^{*-d}(B)$. From the above definition, one can see that fibre integration is a $H^*(B)$-module homomorphism (see also \cite[Prop.\,8.2]{BH58}), i.e.\,for $x\in H^*(B)$ and $y\in H^*(E)$ we have 
\begin{equation}\label{pushpull}
 \int_{\pi}\pi^*(x)\cdot y=x \int_{\pi}y.
\end{equation}

\section{\texorpdfstring{$TM$}{TM}-fibrations}\label{RationalReview}

Denote by $TM$ the tangent bundle of a closed, oriented manifold $M$. Then a $TM$-fibration is a pair $(\pi:E\ra B,T_{\pi}E\ra E)$ consisting of an oriented fibration $\pi:E\ra B$ with fibre $M$ and an oriented vector bundle $T_{\pi}E\ra E$ such that $T_{\pi}E|_{\pi^{-1}(b)}$ is equivalent as a vector bundle to $TM$ for all $b\in B$ (this is a special case of \cite[Def.\,2.1]{Ber20}). We say that two $TM$-fibrations are equivalent if there exists a fibre homotopy equivalence of the underlying fibrations that is covered by a bundle map.

\smallskip
This is precisely the data extracted out of a fibre bundle that we need to define tautological classes (and also the motivation for studying $TM$-fibrations) and we can extend the definition of tautological classes in \eqref{MMM} to $TM$-fibrations. The set of $TM$-fibrations over $B$ up to equivalence is in bijection with $[B,\B\haut^+(TM)]$ by pullback of a univeral $TM$-fibration, where
\begin{equation*}
 \haut^+(TM)=\left \{ \begin{tikzcd}[ampersand replacement=\&,row sep=small,column sep=small] TM \arrow{r}{\bar{f}}\arrow[swap]{r}{\cong} \arrow{d} \& TM\arrow{d}\\
                     M\arrow{r}{f}\arrow[swap]{r}{\simeq}\& M
                    \end{tikzcd}\,\Bigg | \,f\in \haut^+(M),\, \bar{f} \text{ is linear}
\right\}
\end{equation*}
is the monoid of tangential homotopy self-equivalences. Later-on, we focus on the set of path-components $\haut_0(TM)\subset \haut(TM)$ of tangential homotopy equivalences that cover maps in $\haut_0(M)$. The space $\B\haut_0(TM)$ classifies $TM$-fibrations where the underlying fibration has trivial fibre transport.

\smallskip
In the following, we study the subring of $H^*(\B\haut^+(TM))$ generated by MMM-classes. However, we first need to address a subtlety related to the Euler class $e\in H^d(\B\SO(d))$ for even dimensions $d=2n$.

\subsection{The Euler class of \texorpdfstring{$TM$}{TM}-fibrations}
 For even dimensions $d=2n$ the cohomology ring $H^*(\B\SO(d);\Q)$ is the polynomial ring $\Q[p_1,\hdots,p_{n-1},e]$ on the Pontrjagin classes and the Euler class. There are two different ways to associate an Euler class to a $TM$-fibration $(\pi:E\ra B,T_{\pi}E\ra E)$. One can either consider the Euler class of the oriented vector bundle $e(T_{\pi}E)\in H^d(E)$ or the fibrewise Euler class $e^{\fw}(\pi)\in H^d(E)$ introduced in \cite[Def.\,3.1.1]{HLLRW17} that just depends on the underlying fibration $\pi:E\ra B$. For fibre bundles, these two constructions agree but they are different for $TM$-fibrations in general \cite[Prop.\,5.1.5]{Pri20}. Hence, we either have to decide which Euler class we use in the definintion of the $\kappa$-classes, or we can study the classifying space of $TM$-fibrations that satisfy $e(T_{\pi}E)=e^{\fw}(\pi)$. We choose the later option in this paper.
 
 \smallskip
 The classifying space of oriented $M$-fibrations $\pi:E\ra B$ with a given cohomology class $e\in H^d(E)$ that restricts to $e(M)\in H^d(M)$ has been described by May \cite[Sect.\,11]{May75} via the two-sided bar construction and is given by $\B(\Map(M,K(\Z,d))_{e(M)},\haut^+(M),*)$. There are classifying maps 
 \begin{align*}
  \B\haut^+(M)&\xrightarrow{\,e^{\fw}(\pi)\,} \B(\Map(M,K(\Z,d))_{e(M)},\haut^+(M),*),\\
  \B\haut^+(TM)&\xrightarrow{\,e(T_{\pi}E)\,} \B(\Map(M,K(\Z,d))_{e(M)},\haut^+(M),*),
 \end{align*}
 which record the underlying $M$-fibration together with fibrewise Euler class and respectively the Euler class of the vector bundle $T_{\pi}E$.
 \begin{defn}\label{e=efw}
  We define $\B\haut(TM)^{e=e^{\fw}}$ as the homotopy equalizer of 
  \begin{equation*}
  \begin{tikzcd}[ampersand replacement=\&]
     \B\haut^+(M)\times \B\haut^+(TM) \arrow[shift left=1]{r}{e^{\fw}(\pi)\circ \pi_1}\arrow[shift right=1,swap]{r}{e(T_{\pi}E)\circ \pi_2} \& \B(\Map(M,K(\Z,d))_{e(M)},\haut^+(M),*).
  \end{tikzcd}
 \end{equation*}
 \end{defn}
 Then $\B\haut(TM)^{e=e^{\fw}}$ is the classifying space of $TM$-fibrations that satisfy $e^{\fw}(\pi)=e(T_{\pi}E)$. Denote by $(E\overset{\pi}{\ra} \B\haut(TM)^{e=e^{\fw}},T_{\pi}E\ra E)$ the universal $TM$-fibration over $\B\haut(TM)^{e=e^{\fw}}$. We can now define define $\kappa$-classes unambiguously as before, i.e.\,for $c\in H^{|c|}(\B\SO(d))$ we define 
 \begin{equation}
  \kappa_c:=\int_{\pi}c(T_{\pi}E)\in H^{|c|-d}(\B\haut(TM)^{e=e^{\fw}}).
 \end{equation}
 
\begin{defn}
The \emph{homotopical tautological ring} $R^*_h(M)$ is the subring of $ H^*(\B\haut(TM)^{e=e^{\fw}})$ generated by all $\kappa$-classes. 
\end{defn}

For smooth fibre bundles, the definition of the MMM-classes of a smooth $M$-bundle only depends on the underlying $TM$-fibration. Moreover, since the two Euler classes agree there is a map $\B\Diff^+(M)\ra \B\haut(TM)^{e=e^{\fw}}$ which induces a surjection 
\begin{equation}\label{ComparisonTautological}
R_h^*(M)\twoheadrightarrow R^*(M)
\end{equation}
as observed in \cite[Thm 2.10]{Ber20}.
\begin{rem}
  Berglund introduces in \cite{Ber20} a more general version of the homotopical tautological ring for a bundle $\xi$ over $M$ that he denotes by $R^*(\xi)\subset H^*(\B\haut(\xi))$. We have chosen a different notation than $R^*(TM)$ as the definitions are slightly different.
\end{rem}   

 In this sense, the homotopical tautological ring provides an upper bound on $R^*(M)$. The key point is that $R_h^*(M)$ is more amenable to tools from homotopy theory and studying it provides a systematic way to get information about tautological rings as in (i) in the introduction.
 
\bigskip
In the next sections, we use relations from the family signature theorem and tools from rational homotopy theory to study the (homotopical) tautological ring. In general, this only provides information about the image of $R^*(M)$ in $H^*(\B\Diff_0(M))$, which we denote by $R^*_0(M)$, and the image of $R_h(M)$ in $H^*(\B\haut_0(TM)^{e=e^{\fw}})$, which we denote by $R_{h,0}(M)$, and where $\B\haut_0(TM)^{e=e^{\fw}}$ is the classifying space of $TM$-fibrations satisfying $e^{\fw}(\pi)=e(T_{\pi}E)$ and where the underlying fibration has trivial fibre transport. A model of $\B\haut_0(TM)^{e=e^{\fw}}$ is defined analogously as in Definition \ref{e=efw} as the homotopy equalizer of 
   \begin{equation}\label{HEmodel}
  \begin{tikzcd}[ampersand replacement=\&]
      \B\haut_0(M)\times \B\haut_0(TM) \arrow[shift left=1]{r}{e^{\fw}(\pi)\circ \pi_1}\arrow[shift right=1,swap]{r}{e(T_{\pi}E)\circ \pi_2} \& \B(\Map(M,K(\Z,d))_{e(M)},\haut_0(M),*).
  \end{tikzcd}
 \end{equation}
 We can upgrade the computation of $R^*_{h,0}(M)$ if the group $\mathcal{E}(M)$ of homotopy classes of homotopy self-equivalences is finite.
  \begin{lem}\label{Efinite}
   If $\mathcal{E}(M)$ is finite, the inclusion $ \B\haut_0(TM)^{e=e^{\fw}}\ra \B\haut^+(TM)^{e=e^{\fw}} $ induces an isomorphism of homotopical tautological rings $ R^*_{h}(M)\cong R^*_{h,0}(M)$.
  \end{lem}
 \begin{proof}
  By \cite[Cor.\,2.4]{Ber17} there are equivalences
  \begin{equation}\label{Bhaut(TM)}
  \begin{split}
   \B\haut^+(TM)&\simeq \B\left(\Map(M,\B\SO(d))_{TM},\haut(M)_{TM},*\right)\\
   \B\haut_0(TM)&\simeq \B\left(\Map(M,\B\SO(d))_{TM},\haut_0(M),*\right),
   \end{split}
  \end{equation}
  where $\haut(M)_{TM}\subset \haut^+(M)$ denotes path components of homotopy equivalences that preserve the tanget bundle. It follows that the homotopy fibre of
  \begin{equation}\label{inclusion}
   \B\haut_0(TM)^{e=e^{\fw}}\ra \B\haut^+(TM)^{e=e^{\fw}}
  \end{equation}
 is equivalent to $\pi_0(\haut(TM)_{TM})$ and hence is homotopy finite if $\mathcal{E}(M)=\pi_0(\haut(M))$ is finite. Therefore, the comparison map \eqref{inclusion} has the homotopy type of a finite sheeted covering and thus induces an injection on rational cohomology that restricts to an isomorphism on homotopical tautological rings.
 \end{proof}

\subsection{Rational models and fibre integration}

Berglund has determined in \cite{Ber17} a rational model for $\B\haut_0(TM)$ and cocycle representatives of the Euler and Pontrjagin clases in terms of these models in \cite{Ber20}. Based on his models, we determine a model for $\B\haut_0(TM)^{e=e^{\fw}}$. In combination with the work of the author in \cite{Pri19} on the fibrewise Euler class, one can then simply compute $R^*_{h,0}(M)\subset H^*(\B\haut_0(TM))^{e=e^{\fw}}$.

\medskip
We begin by giving an alternative description of $\B\haut_0(TM)^{e=e^{\fw}}$ up to rational equivalence.
\begin{lem}\label{Bhaute=efw}
 Let $M^d$ be an even dimensional, closed, oriented manifold. There is a rational equivalence
 \[\B\haut_0(TM)^{e=e^{\fw}}\xrightarrow{\,\simeq_{\Q}\,} \left(\prod_{i=1}^{d/2-1}\Map(M,K(\Q,4i))_{p_i(M)}\right)\ssslash \haut_0(M),\]
 where $\Map(M,K(\Q,4i))_{p_i(M)}$ denotes the path component the rational Pontrjagin class $p_i(M)\in H^{4i}(M;\Q)$ and $\haut_0(M)$ acts by precomposition.
\end{lem}
\begin{proof}
 Consider the spaces $M_P=\prod_{i=1}^{d/2-1}\Map(M,K(\Q,4i))_{p_i(M)}$ and $M_E=\Map(M,K(\Q,d))_{e(M)}$ that record the path components of the rational Pontrjagin classes and the rational Euler class respectively. The homotopy quotient $M_P\sslash \haut_0(M)$ is the classifying space of $M$-fibrations with trivial fibre transport and rational cohomology classes on the total space that restrict to the Pontrjagin classes of $M$ on the fibre (and similarly for $M_E$). Consider the square
\begin{equation}
\begin{tikzcd}[ampersand replacement=\&]
M_P\sslash\haut_0(M) \arrow{r}\arrow{d} \& \B\haut_0(M)\arrow{d}{e^{\fw}}\\
(M_P\times M_E)\sslash \haut_0(M)\arrow{r}{\pi_{M_E}} \& M_E\sslash \haut_0(M)
\end{tikzcd}
\end{equation}
where the horizontal maps are induced from projections and the vertical maps are given by the classifying map of the fibration together with the rational fibrewise Euler class. Since the square commutes up to homotopy and the horizontal homotopy fibres agree, it follows that $M_P\sslash \haut_0(M)$ is the homotopy equalizer of 
   \begin{equation}\label{QHEmodel}
  \begin{tikzcd}[ampersand replacement=\&]
      \left((M_P\times M_E)\sslash \haut_0(M)\right) \times \B\haut_0(M) \arrow[shift left=1]{r}{e^{\fw}(\pi)\circ \pi_2}\arrow[shift right=1,swap]{r}{\pi_{M_E}\circ \pi_1} \& M_E\sslash \haut_0(M).
  \end{tikzcd}
 \end{equation}
 Since $\Map(M,\B\SO(d))_{TM}\simeq_{\Q} M_P\times M_E$ and $\Map(M,K(\Z,d))_{e(M)}\simeq_{\Q}M_E$, there are rational equivalences
 \begin{align*}
 \B\haut_0(TM)\overset{\eqref{Bhaut(TM)}}{\simeq} \Map(M,\B\SO(d)_{TM}\sslash \haut_0(TM) &\overset{\simeq_{\Q}}{\lra}(M_P\times M_E)\sslash \haut_0(M),\\
 \Map(M,K(\Z,d))_{e(M)}\sslash \haut_0(M)&\overset{\simeq_{\Q}}{\lra} M_E\sslash \haut_0(M),
 \end{align*}
 and we obtain a map of diagrams from \eqref{HEmodel} to \eqref{QHEmodel} which is a rational homotopy equivalence on the underlying spaces. Therefore, the induced map on homotopy equalizers 
 \[ \B\haut_0(TM)^{e=e^{\fw}}\xrightarrow{\,\simeq_{\Q}\,}M_P\sslash \haut_0(M).
 \]
 is a rational homotopy equivalence.
\end{proof}
We can now apply Berglund's results \cite{Ber17,Ber20} in combination with \cite{Pri19} to obtain models for $\B\haut_0(TM)^{e=e^{\fw}}$ as well as representatives of the characteristic classes of $T_{\pi}E\ra E$. 
In the following, we consider manifolds $M$ homotopy equivalent to $\HH P^2$ and we denote a minimal Sullivan model of $M$ by $\Lambda=(\Lambda(x,y),d=x^3\frac{\partial}{\partial y})$ with $|x|=4$ and $|y|=11$ (where $x\in \Lambda $ represents a generator of the integral cohomology ring).
\begin{prop}\label{modelE=Efw}
	Let $M$ be a manifold homotopy equivalent to $\HH P^2$ with cocycle representatives of the Pontrjagin classes given by $p_i(\Lambda)=p_ix^i\in \Lambda$ for $p_i\in \Q$ for $i\in \{1,2\}$. The cohomology ring $H^*(\B\haut_0(TM)^{e=e^{\fw}})$ is a polynomial ring \[B:=\Q[a_8,a_{12},p_{1,0},p_{2,0},p_{2,1},p_{3,0},p_{3,1},p_{3,2}]\]
 	on generators of degree $8,12,4,8,4,12,8,4$, and the $TM$-fibration $E\ra \B\haut_0(TM)^{e=e^{\fw}}$ is formal with algebraic model given by
	\[B\lra E:=B[x]/(x^3+a_8x+a_{12}),
	\]
	and representatives of the characteristic classes of $T_{\pi}E\ra E$
	\begin{equation}\label{RepClasses}
	\begin{aligned}
	e(T_{\pi}E)&=3x^2+a_8, \\ p_1(T_{\pi}E)&=p_1x+p_{1,0},
	\end{aligned}
	\qquad
	\begin{aligned}  p_2(T_{\pi}E)&=p_2x^2+p_{2,1}x+p_{2,0},\\
	 p_3(T_{\pi}E)&=p_{3,2}x^2+p_{3,1}x+p_{3,0}.
	\end{aligned}
	\end{equation}
\end{prop} 
 \begin{proof}
Denote by $\Pi$ be the graded vector space $\Q\{q_1,q_2,q_3\}$ where $|q_i|=4i-1$, which is a dg Lie algebra with trivial differential and bracket. Then $\tau(M):=\sum_{i=1}^{3}p_i(\Lambda)\otimes q_i \in \MC(\Lambda\hat{\otimes}\Pi)$ and the twisted dg Lie algebras $(\Lambda\hat{\otimes}\Pi)^{\tau(M)}\langle 0 \rangle$ is a dg Lie model for $M_P=\prod_{i=1}^3\Map(M,K(\Q,4i)_{p_i(M)}$ by \cite{Be15}.

Since $\Der^+(\Lambda)$ acts on $\Lambda\hat{\otimes}\Pi$ through derivations, we can consider the semi-direct product 
 	\begin{equation}
 	\begin{split}
 	\gL_M:=(\Der^+(\Lambda)\ltimes \Lambda\hat{\otimes}\Pi_d)^{\tau(M)}\langle 0\rangle
 	\end{split}
 	\end{equation}
 	which is a model for $M_P\sslash \haut_0(M)$ for by \cite[Thm 3.8]{Ber20} and hence also for $\B\haut_0(TM)^{e=e^{\fw}}$ by Lemma \ref{Bhaute=efw}.
 	As a chain complex, $\Der^+(\Lambda)$ has a basis given by $\{x^2\frac{\partial}{\partial y},\;\frac{\partial}{\partial x},\;x\frac{\partial}{\partial y},\;\frac{\partial}{\partial y}\}$,
 	where the only non-trivial differential is given by $[d,\frac{\partial}{\partial x}]=-3x^2\frac{\partial}{\partial y}$. Hence the inclusion $\mathfrak{a}\hra \Der^+(\Lambda)$ of the abelian Lie sub algebra $\mathfrak{a}:=\Q\{x \frac{\partial}{\partial y},\,\frac{\partial}{\partial y}\}$ is a quasi-isomorphism, and we obtain a smaller Lie model for $\gL_M$ by
 	\begin{align*}
 	\begin{split}
 	(\mathfrak{a}\ltimes \Lambda\hat{\otimes}\Pi)^{\tau(M)}\langle 0\rangle&\overset{\simeq}{\hra} \gL_M
 	\end{split}
 	\end{align*}
 	We can simplify further by observing that the linear map $H(\Lambda)\ra \Lambda$ which assigns cocycle representatives induces a quasi-isomorphism of dg Lie algebras $H(\Lambda)\otimes \Pi\overset{\simeq}{\ra} \Lambda\hat{\otimes}\Pi$ because $\Pi$ is abelian. Moreover, this quasi-isomorphism is a map of $\mathfrak{a}$-modules with respect to the trivial $\mathfrak{a}$-module structure on $H(\Lambda)\hat{\otimes} \Pi$, and therefore there are quasi-isomorphisms of dg Lie algebras
 	\begin{align*}
 	\begin{split}
 	\mathfrak{a}_M&:=\mathfrak{a}\times (H(\Lambda)\otimes \Pi)\langle 0\rangle \overset{\simeq}{\hra} (\mathfrak{a}\ltimes \Lambda\hat{\otimes}\Pi)^{\tau(M)}\langle 0\rangle \overset{\simeq}{\hra} \gL_M,
 	\end{split}
 	\end{align*}
 	Since $\mathfrak{a}_M$ is abelian with trivial differential, its Chevalley-Eilenberg complex is  $\cC_{CE}(\mathfrak{a}_M)\cong \Q[a_8,a_{12},p_{1,0},p_{2,0},p_{2,1},p_{3,0},p_{3,1},p_{3,2}]$ with trivial differential where the generators correspond to the (suspended) dual basis elements of the following basis of $\mathfrak{a}_M$:
 	\begin{align*}
 	&\big\{x \frac{\partial}{\partial y},\,\frac{\partial}{\partial y}, 1\otimes q_1,\,1\otimes q_2,\,x\otimes q_2,\,1\otimes q_3,\,x\otimes q_3,\,x^2\otimes q_3\big\}.
 	\end{align*}
    
    It follows from Theorem 3.8 in \cite{Ber20} that a relative Sullivan model for the universal $TM$-fibration $E\ra \B\haut_0(TM)^{e=e^{\fw}}$ is given by $\eta_*:\cC_{CE}(\mathfrak{a}_M;\Q)\ra \cC_{CE}(\mathfrak{a}_M;\Lambda)$, where $\mathfrak{a}_M$ acts on $\Lambda$ through the projection onto $\mathfrak{a}\subset \Der^+(\Lambda)$ and where $\eta:\Q\ra\Lambda$ denotes the unit. With the sign convention for the Chevalley-Eilenberg complex in \eqref{CEconvention}, we find that 
    $\cC_{CE}(\mathfrak{a}_M;\Lambda)\cong (\cC_{CE}(\mathfrak{a}_M)\otimes \Lambda (x,y),D(y)=x^3+a_8x+a_{12})$, and hence that $\eta_*$ is equivalent to 
    \[ 
     \cC_{CE}(\mathfrak{a}_M)\lra \cC_{CE}(\mathfrak{a}_M)[x]/(x^3+a_8x+a_{12})
    \]
     via the evident projection. A straight forward computation shows that the cocycle representatives in $\cC_{CE}(\mathfrak{g}_M;\Lambda)$ of the Pontrjagin classes that Berglund determines \cite[Thm 3.8]{Ber20} pull back to 
	\begin{align*}
	 p_1(T_{\pi}E)&=p_1x+p_{1,0}, & p_2(T_{\pi}E)&=p_2x^2+p_{2,1}x+p_{2,0},\\
	  p_3(T_{\pi}E)&=p_{3,2}x^2+p_{3,1}x+p_{3,0}, & & 
	\end{align*}
	under the inclusion $\mathfrak{a}_M\overset{\simeq}{\hra}\gL_M$. Lastly, the fibrewise Euler class of the underlying fibration is represented by $e^{\fw}(\pi)=3x^2+a_8$ by \cite[Thm 5.6]{Pri19} and agrees by construction with $e(T_{\pi}E)$, which concludes the proof.
 \end{proof}
\begin{rem}
 There is a good reason why the rational model for the universal $TM$-fibration with $M\simeq \HH P^2$ above is so simple. Namely, $\HH P^2$ is positively rationally elliptic, i.e.\,the collection of all rational homotopy and cohomology groups is finite dimensional and $\chi(M)>0$. Rational models for such spaces are very rigid \cite[Ch.\,32]{FHT}. Moreover, Halperin's conjecture states that for any fibration with trivial fibre transport and positively rationally elliptic fibre the Serre spectral sequence collapses. Halperin's conjecture is easy to check for particular positively rationally elliptic spaces and consequently, fibrations with rationally elliptic fibres are quite easy to study.
 
 This is in stark contrast to fibrations with rationally hyperbolic fibres $X$, for which the model of $\B\haut_0(X)$ in terms of derivations of a Sullivan model are quite unwieldy. Instead, rational models of fibrations with rationally hyperbolic fibres are more accessible through dg Lie models and turn out to be much more complicated to analyse (see for example \cite{BM14,Sto22}).
\end{rem}

\subsection{Fibre integration}
 Let $E\ra B$ be an $M$-fibration that satisfies the assumption of the Leray-Hirsch theorem, i.e.\,the restriction to the fibre $H^*(E)\ra H^*(M)$ is surjective. This implies that $H^*(E)$ is a free module over the cohomology ring of the base and in this case, fibre integration is completely determined by the push-pull formular \eqref{pushpull}.
 
 \smallskip
 It follows from Proposition \ref{modelE=Efw} that $\{1,x,x^2\}$ is a free $H_{CE}(\mathfrak{a}_M)$-module basis of $H_{CE}(\mathfrak{a}_M;\Lambda)$. Consequently, fibre integration is determined by 
\begin{equation}
\pi_!(x^i)=\begin{cases}
1 & i=2\\
0 & i<2
\end{cases}.
\end{equation}
Fibre integrals of higher powers of $x$ can be computed by rewriting in terms of the module basis. For example, since $x^4+a_8x^2+a_{12}x=0\in H_{CE}(\mathfrak{a}_M;\Lambda)$ we find that $\pi_!(x^4)=-a_8$, and one can easily work out representatives of the $\kappa$-classes.

\section{Interaction with the family signature theorem}\label{signature}

Hirzebruch's signature theorem \cite{Hir66} describes an intricate link between the tangent bundle and the topology of a manifold $M^{4k}$. It states that there are polynomials $L_k\in H^{4k}(\B\SO;\Q)$ such that for closed oriented manifold $M^{4k}$ the signature of the intersection form is given by $\text{sign}(M)=\int_M L_k(TM)$. The first few of these polynomials are given by
\begin{align*}
L_1&=\frac{1}{3}p_1,\\
L_2&=\frac{1}{45}(7p_2-p_1^2).
\end{align*}
There is a fibrewise version of this theorem originally due to Atiyah \cite{At69} for bundles $\pi:E\ra B$ with fibre $M$ that relates $\kappa_{L_i}=\int_{\pi}L_i(T_{\pi} E)$ with certain invariants of fibrations associated to the signature of the local system of $\mathcal{H}^{d/2}(M;\Q)$ if $d$ is even. More precisely, there are classes $\sigma_{4i-d}\in H^{4i-d}(\B\Aut(H_M,\lambda);\Q)$,  where $\Aut(H_M,\lambda)$ denotes the group of automorphisms of $H^{d/2}(M;\Q)$ which preserve the intersection pairing $\lambda$, whose pullback along the map $\phi:\B\haut^+(M)\ra \B\Aut(H_M,\lambda)$ which classifies the local system $\mathcal{H}^{d/2}(M;\Q)$ associated to the bundle $\pi:E\ra B$ agree with $\kappa_{L_i}$.
\begin{thm}[{\cite[Thm 2.6]{RW22}}]\label{FamilySignatureTheorem}
	Let $\pi:E\ra B$ be a smooth, oriented fibre bundle with fibre $M^d$ a closed, oriented manifold. Then
	\begin{equation}
	\kappa_{L_i}=\int_{\pi}L_i(T_{\pi}E)=\begin{cases}
	\phi^*(\sigma_{4i-d}) & \text{if d is even}\\
	0 & \text{if d is odd}
	\end{cases}
	\end{equation}
\end{thm}
\begin{rem}
 Note that the statement above requires rational coefficients. Aside from more refined statements of the family signature theorem, Randal-Williams has also shown in \cite{RW22} that Theorem \ref{FamilySignatureTheorem} holds more generally for topological block bundles, so the family signature theorem depends on the bundle but not the smooth structure.
\end{rem}

This theorem has first been exploited to study tautological rings in \cite[Thm 2.1]{GGRW17}, where they observe that since the automorphism group $\Aut(H_M,\lambda)$ is arithmetic, its classifying space has finite rational cohomological dimension so that $\kappa_{L_i}=\phi^*(\sigma_{4i-d})$ must be nilpotent.  

\smallskip
In our case, we restrict our attention to the path-component of the identity $\Diff_0(M)$ (resp.\,$\haut_0(M)$) so that $\phi:\Diff_0(M)\subset \haut_0(M)\ra \Aut(H_M,\lambda)$ is constant and hence $\kappa_{L_i}=0\in R_0(M)$. The family signature theorem does not hold, however, for $TM$-fibrations in general (as we will see below). Thus we can improve our upper bound $R_{h,0}^*(M)$ by enforcing these additional relations in the homotopical tautological ring. 
\begin{defn}
	Let $i^*L_k\in H^{4k}(\B\SO(d))$ denote the restriction of the L-classes and define $\kappa_{L_k}=\int_{\pi}i^*L_k(T_{\pi}E)\in H^{4k-d}(\B\haut_0(TM)^{e=e^{\fw}})$. The \emph{Hirzebruch ideal} $I_H\subset H^{*}(\B\haut_0(TM)^{e=e^{\fw}};\Q)$ is the ideal generated by all $\kappa_{L_k} $ for $4k>\dim\,M$. For computational reasons, we also consider the ideal $I_H^{\leq n}\subset I_H$ generated by $\kappa_{L_k}$ for $\dim M/4<k\leq n$. 
\end{defn}

By the family signature theorem, the surjection \eqref{ComparisonTautological} factors as
	\begin{equation}\label{surjection}
	R^*_{h,0}(M)\lra R^*_{h,0}(M)/(I_H^{\leq k}\cap R^*_{h,0}(M))\lra R^*_{h,0}(M)/(I_H\cap R^*_{h,0}(M)) \lra R^*_0(M), 
	\end{equation}
which is a better upper bound than $R^*_{h,0}(M)$. In particular, we obtain an upper bound on the Krull dimension
\begin{equation}\label{UpperBound}
\Kdim R_{0}^*(M)\leq \Kdim  R^*_{h,0}(M)/(I_H\cap R^*_{h,0}(M)).
\end{equation}
In practice, it turns out to be much simpler to compute the dimension of the Hirzebruch ideal $I_H\subset H^*(\B\haut_0(TM)^{e=e^{\fw}};\Q)$ which also provides an upper bound on the Krull dimension in some cases of interest.

\begin{lem}\label{dimI}
 If $M$ is even dimensional and $R_{h,0}^*(M)$ and $H^{\text{even}}(\B\haut_0(M)^{e=e^{\fw}})$ are finitely generated as $\Q$-algebras, then $\Kdim R_0^*(M)\leq \dim I^{\leq k}_H $ for all $k$.
\end{lem}
\begin{proof}
For finitely generated $k$-algebras the transcendence degree coincides with the Krull dimension \cite[Thm 5.9]{Kem11}. In particular, if $A\subset B$ are finitely generated $k$-algebras then $\Kdim A\leq \Kdim B$. Since both $ R^*_{h,0}(M)/(I^{\leq k}_H\cap R^*_{h,0}(M))\subset H^{\text{even}}(\B\haut_0(M)^{e=e^{\fw}})/I_H^{\leq k} $ are finitely generated by assumption, we obtain bounds on the Krull dimensions 
\[ \Kdim R_0^*(M)\overset{\eqref{UpperBound}}{\leq} \Kdim  R^*_{h,0}(M)/(I_H^{\leq k}\cap R^*_{h,0}(M))\leq \Kdim H^{\text{even}}(\B\haut_0(M)^{e=e^{\fw}})/I_H^{\leq k}. \]
\end{proof}
One can compute $\dim I_H^{\leq k}$ for a fixed value of the first Pontrjagin class using Macaulay2 \cite{M2}; in the table below we state a few sample computations for $k=10$:
\begin{equation}\label{dimIH}
         \begin{tabular}{ c| c c c}
              $p_1(M)\in \Lambda$ & $2x$ & $3x$ & $4x$  \\\hline
            $\dim I_H^{\leq 10}$ & 3 & 0 & 0 \\
        \end{tabular}
\end{equation}
In the next section, we discuss a strategy to treat all values of $p_1(M)$ at the same time.

\subsection{Parametrized Hirzebruch ideal}\label{ParametrizedHirzebruchSection}
The main new idea in this paper is to treat the values of the Pontrjagin classes as additional parameters subject to the relation enforced by the signature theorem. We can then study the quotient $B/I_H$ and apply a version of Grothendieck's generic freeness theorem to obtain information about the dimension of $B/I_H$ as we vary the Pontrjagin classes.

\smallskip
In the following, we demonstrate this idea in the case $M\simeq \HH P^2$. For $i=1,2$ denote by $p_i(\Lambda)=p_ix^i\in \Lambda$ with $p_i\in \Q$ representatives of the two nontrivial Pontrjagin classes. By the signature theorem $p_2=\frac{1}{7}(45+p_1^2)$, so that we add one additional parameter $P_1$ of degree $0$ that parametrizes the value of the first Pontrjagin class to our model in Proposition \ref{modelE=Efw}
\begin{equation}\label{ParametrizedUniversalFibration}
\overline{B}:=\Q[P_1][a_8,a_{12},p_{1,0},p_{2,0},p_{2,1},p_{3,0},p_{3,1},p_{3,2}]\lra \overline{E}:=\overline{B}[x]/(x^3+a_8x+a_{12})
\end{equation}
with elements in $\overline{E}$
\begin{align*}
  e^{\fw}(\pi)&=3x^2+a_8, & p_2(T_{\pi}E)&=\frac{1}{7}(45+P_1^2)x^2+p_{2,1}x+p_{2,0},\\
 p_1(T_{\pi}E)&=P_1x+p_{1,0}, & p_3(T_{\pi}E)&=p_{3,2}x^2+p_{3,1}x+p_{3,0},
\end{align*}
that extend the representatives of the characteristic classes of $T_{\pi}E\ra E$ with the additional parameter $P_1$. Denote by $\overline{I}_H\subset \overline{B}$ the Hirzebruch ideal generated by $\kappa_{L_i}\in \overline{B}$ for $i>2$, and consider the map of rings 
\begin{equation}\label{ParametrizedHirzebruch}\va:\Q[P_1]\ra \overline{B}/\overline{I}_H.\end{equation}
We can recover the quotient of the Hirzebruch ideal for a specific value $p_1\in \Q$ of the first Pontrjagin class $p_1(M)=p_1x\in \Lambda$ from the morphism $\va$ in \eqref{ParametrizedHirzebruch} as $\Q[P_1]/(P_1-p_1)\otimes _{\Q[P_1]}\overline{B}/\overline{I}_H$, where $(P_1-p_1)\subset \Q[P_1]$ denotes the maximal ideal generated by $P_1-p_1\in \Q[P_1]$ and where $\overline{B}/\overline{I}_H$ is a $\Q[P_1]$-algebra via $\va$.

\smallskip 
More generally, given a homomorphism $\va:R\ra S$ of Noetherian rings such that $S$ is a finitely generated $R$-algebra, one can study how the fibres $K(R/P)\otimes_RS$ vary as we vary the prime $P\subset R$ (where $K(R/P) $ denotes the field of fractions of $R/P$). Grothendieck's generic freeness lemma implies that over some open set of $\text{Spec}(R)$, the fibres share some common properties.

The precise statement we use is given in \cite[Thm 14.8b]{Eis95}, and states that if $S=\bigoplus_{i\geq 0} S_0$ is a positively graded algebra, finitely generated over $R=S_0$, then for each integer $e$ there is an ideal $J_e\subset R$ so that for any prime ideal $P\subset R$, the following are equivalent
\[\Kdim K(R/P)\otimes _RS \geq e \quad \text{if and only if} \quad P\supset J_e.\]
This can be applied to the map $\va$ in \eqref{ParametrizedHirzebruch} and with the aide of computer computations we find the following.
\begin{thm}\label{AlgebraicMain}
 The Krull dimension of the fibre of $\va$ over maximal ideals $(P_1-p_1)\subset \Q[P_1]$ vanishes for all but finitely many $p_1\in \Q$.
\end{thm}
\begin{proof}
  One can compute the dimension of the fibres of $\va:\Q[P_1]\ra \overline{B}/\overline{I}_H^{\leq 10}$ over the maximal ideals $(P_1-p_1)$ for particular values of $p_1\in \Q$ using Macaulay2 and we have found that for certain values of $p_1$ that the dimension of the fibre vanishes (see \eqref{dimIH}). 
  By \cite[Thm 14.8b]{Eis95} stated above, it follows that for $e=1$ the ideal $J_1$ is not zero, and thus only finitely many maximal ideals $(P_1-p_1)$ can contain $J_1$. Hence, for almost all values of $p_1$ the dimension of the fibre vanishes, and since $\dim I_H\leq \dim I_H^{\leq 10}$ the statement follows.
\end{proof}

Our main Theorem can be deduced from this algebraic statement. There are infinitely many fake quaternionic spaces $M\simeq \HH P^2$ \cite{Hsi66,EKu62} that are classified up to diffeomorphisms by the Pontrjagin number $p_1(M)^2$ \cite{KrSt07}. In our main theorem, we show that the Krull dimension of $R^*(M)$ vanishes for all fake quaternionic spaces except for a finite number of cases.

\begin{thm}\label{Main}
	For almost all manifolds $M\simeq \HH P^2$, we have $\Kdim R^*(M)=0$..
\end{thm}
\begin{proof}
 The group of homotopy self-equivalences $\mathcal{E}(M)$ is finite \cite[Sect.\,9]{Ba96} and hence the map $\B\haut_0(M)^{e=e^{\fw}}\ra \B\haut(M)^{e=e^{\fw}}$ induces an isomorphism $R^*_h(M)\overset{\cong}{\lra} R^*_{h,0}(M)$ by Lemma \ref{Efinite} and consequently also of Hirzebruch ideals. Moreover, since the automorphism group $\Aut(H_M,\lambda)$ is finite as well, $\kappa_{L_i}\in R^*(M)$ is not just nilpotent but vanishes. Therefore the factorisation \eqref{surjection} lifts to 
 \[R_h^*(M)\twoheadrightarrow R_h^*(M)/(R^*_h(M)\cap I_H) \twoheadrightarrow R^*(M).\]
 We can then apply Theorem \ref{AlgebraicMain} and Corollary \ref{dimI} to $R_h^*(M)/(R^*_h(M)\cap I_H)$ to see that the Krull dimension of the quotient ring $R_h^*(M)/(R^*_h(M)\cap I_H)$ vanishes for all but finitely many values of the first Pontrjagin class and hence so does $\Kdim R^*(M)$.
\end{proof}
\begin{rem}
 The proof of Theorem \ref{Main} works more generally for 1-connected, closed oriented manifolds $M$ that are rationally homotopy equivalent to $\HH P^2$ if the dimension of the fibre of \eqref{ParametrizedHirzebruch} vanishes for $p_1(M)$. Just that it is less clear what \emph{almost all} means in this context.
\end{rem}
\section{Torus actions and the \texorpdfstring{$\hat{A}$}{A}-genus}\label{endsection}
Theorem \ref{Main} begs the question whether we can determine precisely the set of all fake quaternionic spaces $M$ with $\Kdim R^*(M)>0$. Or relatedly, if we can determine the locus of the exceptional fibres of $\va$ in \eqref{ParametrizedHirzebruch}, i.e.\,the values of $P_1$ for which the fibres have a positive dimension. We address these questions in this last section.

\subsection{Tautological rings and torus actions} 

The most effective method we currently have to establish lower bounds on $\Kdim R^*(M)$ is using group actions on $M$, and the most streamlined and general statement is again due to Randal-Williams \cite{RW16}. Let $M$ be a manifold with a smooth torus action, i.e.\,a continuous group homomorphism $T^k\ra \Diff_0(M)$, and consider the associated smooth $M$-bundle given by the Borel construction $\pi:M\sslash T^k\ra *\sslash T^k$. The cohomology ring of $*\sslash  T^k=\B T^k$ is a polynomial ring $H^*_T:=\Q[x_1,\hdots,x_k]$ and we denote the tautological ring of $\pi$ by $R^*_T(M)\subset H^*_T$. One of the main results of Randal-Williams \cite[Thm\,3.1]{RW16} provides a criterion on the action that guarantees that $H^*_T$ is integral over $R^*_T(M)$ and hence $\Kdim R^*(M)\geq \Kdim R^*_T(M)=k$.  We state a special case of this theorem below.
\begin{thm}[{\cite[Cor.\,B]{RW16}}]\label{LowerBoundKdim}
	Let $T^k$ act effectively on a smooth manifold $M$ such that the fixed set $W^T$ is discrete and non-empty. Then $H^*_T$ is integral over $R^*_T(M)$ and therefore $\Kdim R^*(M)\geq k$.
\end{thm}

$\HH P^2$ admits a $T^2\subset (\C^{\times})^2$ action given by $(\lambda_1,\lambda_2)\cdot [h_0{:}h_1{:}h_2]=[h_0{:}\lambda_1h_1{:}\lambda_2h_2]$ with isolated fixed points $[1{:}0{:}0]$, $[0{:}1{:}0]$ and $[0{:}0{:}1]$, and so Randal-Williams' theorem implies that $\Kdim R^*(\HH P^2)\geq 2$. We obtain an upper bound on $\Kdim R^*(\HH P^2)$ by Lemma \ref{dimI} by the dimension the fibre of $\va$ in \eqref{ParametrizedHirzebruch} for $p_1=2$ (the value of the first Pontrjagin class of $\HH P^2$) and which we computed with Macaulay2 to be $\Kdim \Q[P_1]/(P_1-2)\otimes_{\Q[P_1]}\overline{B}/\overline{I}_H=3$. 

\begin{prop}
	$2\leq \Kdim R^*(\HH P^2)\leq 3$.
\end{prop}

It turns out that this is the only case where we can find lower bounds on $\Kdim R^*(M)$ via torus actions by the following result of Atiyah and Hirzebruch.
 \begin{thm}[\cite{AH70}]\label{ahat}
 Let $M^{4k}$ be a compact oriented smooth manifold with $w_2(M)=0$. If a connected compact Lie group $G$ acts non-trivially on $M$ then $\hat{A}(M)=0$.
 \end{thm}
 The $\hat{A}$-genus is the ring homomorphism $\hat{A}:\Omega^{\SO}_*\otimes \Q\ra \Q$ associated to the multiplicative sequence induced by the power series of the function $\frac{x/2}{\text{sinh}(x/2)}$. Its first two elements as polynomials in the Pontrjagin classes are given by 
 	\begin{align*}
 		\hat{A}_1(p_1)&=-\frac{1}{24}p_1,\\
 		\hat{A}_2(p_1,p_2)&=\frac{1}{5760}(-4p_2+7p_1^2),
 	\end{align*}
 and the $\hat{A}$-genus of a manifold $M^{4k}$ is defined as $\hat{A}(M^{4k}):=\int_M\hat{A}_k(M)$.
 
\smallskip
 Now consider a fake quaternionic space $M\simeq \HH P^2$ that satisfies $\hat{A}(M)=0$. By the signature theorem, we have that $L(M)=\int_ML_2(M)=1$, and the only possible solutions to these two equations are $p_1(M)=\pm 2x$. By the classification of fake quaternionic manifolds \cite[Thm 1.3]{KrSt07}, $\HH P^2$ (with the usual smooth structure) is the only one with this value of the first Pontrjagin class, and thus a fake quaternionic manifold $M$ only admits a nontrivial torus action if $M\cong \HH P^2$. We suspect that $\HH P^2$ is also the only exception to Theorem \ref{Main}.
 
 \subsection{Homotopical tautological rings and the \texorpdfstring{$\hat{A}$}{A}-genus}
 
 Similarly, we can ask for the exceptions to Theorem \ref{AlgebraicMain}, i.e.\,the values of $p_1\in \Q$ for which the dimension of $\Q[P_1]/(P_1-p_1)\otimes_{\Q[P_1]}\overline{B}/\overline{I}_H$ is positive (which is a necessary condition for $\Kdim R^*(M)>0$). This problem is completely algebraic and equivalent to determining the ideal $J_1$ from \cite[Thm 14.8b]{Eis95} that we used in the proof above and which appears difficult to compute.  
 
 \medskip
 
 One interesting observation, however, that we discuss in this last section is that the vanishing of the $\hat{A}$-genus provides a completely algebraic criterion to distinguish the values of the first Pontrjagin class for which the dimension of the corresponding Hirzebruch ideal is positive. This relationship is completely unclear from the purely commutative algebra perspective, and we present some evidence that this is not merely a coincidence, and that the $\hat{A}$-genus is a suitable algebraic substitute for the geometric concept of a circle action to obtain lower bounds on the Krull dimension of the homotopical tautological ring.
 
 \smallskip
 
 To that end, we revisit the construction of the parametrized Hirzebruch ideal in \eqref{ParametrizedUniversalFibration}, but now we add two formal parameters $P_1$ and $P_2$ that parametrize the value of $p_1(M)$ and $p_2(M)$. More precisely, we consider 
 \begin{equation}
\overline{B}:=\Q[P_1,P_2][a_8,a_{12},p_{1,0},p_{2,0},p_{2,1},p_{3,0},p_{3,1},p_{3,2}]\lra \overline{E}:=\overline{B}[x]/(x^3+a_8x+a_{12}),
\end{equation}
where $P_1,P_2$ are formal parameters of degree $0$, with elements in $\overline{E}$
\begin{align*}
  e^{\fw}(\pi)&=3x^2+a_8, & p_2(T_{\pi}E)&=P_2x^2+p_{2,1}x+p_{2,0},\\
 p_1(T_{\pi}E)&=P_1x+p_{1,0}, & p_3(T_{\pi}E)&=p_{3,2}x^2+p_{3,1}x+p_{3,0},
\end{align*}
that represent the characteristic classes.
\begin{rem}
 The difference to what we have done in the previous section is that we do not enforce the signature theorem $45=7P_2-P_1^2$ in this algebraic model. This is becaus we are interested in this purely algebraic problem and its relation to the $\hat{A}$-genus detached from the geometry of manifolds.
\end{rem}

By abuse of notation, we again denote by $\overline{I}_H$ the parametrized Hirzebruch ideal generated by $\kappa_{L_i}\in \overline{B}$ for $i>2$ in $\overline{B}$, and consider the map
\[\tilde{\va}: \Q[P_1,P_2]\lra \overline{B}/\overline{I}_H.\]
The dimension of the fibre of $\tilde{\va}$ over the maximal ideal $(P_1-p_1,P_2-p_2)\subset \Q[P_1,P_2]$ for $p_1,p_2\in\Q$ corresponds to the dimension of the quotient of $B/I_H$ for that particular choice for the value of the first and second Pontrjagin class (although there is not necessarily a manifold with these Pontrjagin classes). The argument of Theorem \ref{AlgebraicMain} still applies to show that the generic fibre have vanishing dimension, but the exceptional fibres are now not isolated points in $\mathbb{A}^2$ but potentially could be given by a one-dimensional variety.

\medskip
We have computed the dimension of the fibre of $\tilde{\va}:\Q[P_1,P_2]\ra\overline{B}/\overline{I}_H^{\leq 10}$ for several points in $\mathbb{A}^2$ using Macaulay2. The interesting observation is that for all examples we checked for which the $\hat{A}$-genus vanishes, i.e.\,points on $V(-4P_2+7P_1^2)\subset \mathbb{A}^2$, the dimension of the fibre of $\tilde{\va}$ is at least $1$. So even though this question is detached from the existence of circle actions on manifolds, the vanishing of the $\hat{A}$-genus of this algebraic family still seems to provide a positive lower bound on the Krull dimension of $B/I_H$ for the corresponding value of the Pontrjagin classes. We do not have an explanation for this observation but it hints at a connection of the vanishing of the $\hat{A}$-genus and lower bounds on the Krull dimension of homotopical tautological rings.

\medskip
	We have summarized our sample computations in the following picture of $\mathbb{A}^2$, where every circle represents a point in the affine plane for which the Krull dimension of the fibre of $\tilde{\va}:\Q[P_1,P_2]\ra \overline{B}/\overline{I}^{\leq 10}_{H} $ is $1$, and every square corresponds to a point where the dimension of the fibre is $3$. We have also indicated the varieties $V(L_2)$ and $V(\hat{A}_2)$, where we expect the dimension of the fibre to be positive, and for completeness also $V(L_2-1)$ which contains the values of $p_1$ and $p_2$ of fake quaternionic projective spaces by the signature theorem.
	
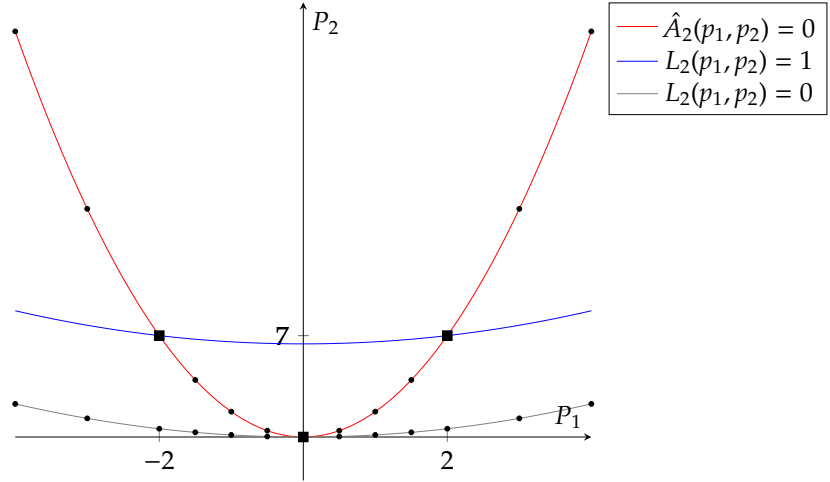
\begin{center}
		\begin{tikzpicture}[scale=0.9]
		\begin{axis}[
		axis lines = middle,
		xlabel = $P_1$,
		ylabel = {$P_2$},
		xmin=-4, xmax=4,
		ymin=-3, ymax=30,
		xtick={-2,2},
		ytick={7,7},
		legend pos=outer north east,
		]
		\addplot [
		domain=-4:4, 
		samples=100, 
		color=red,
		]
		{7*x^2/4};
		\addlegendentry{$\hat{A}_2(p_1,p_2)=0$}
		\addplot [
		domain=-4:4, 
		samples=101, 
		color=blue,
		]
		{(45+x^2)/7};
		\addlegendentry{$L_2(p_1,p_2)=1$}

		\addplot [
		domain=-4:4, 
		samples=1000, 
		color=gray,
		]
		{x^2/7};
		\addlegendentry{$L_2(p_1,p_2)=0$}
		\addplot[
		only marks,    
		color=black,
		mark=square*,
		mark size=2pt,
		]
		coordinates {(0,0)(2,7)(-2,7)};
		\addplot[
		only marks,    
		color=black,
		mark=*,
		mark size=1pt,
		]
		coordinates {(1,7/4)(-1,7/4)(3,7*9/4)(4,28)(-3,7*9/4)(-4,28)(1/2,7/16)(3/2,7*9/16)(-1/2,7/16)(-3/2,7*9/16)};
		
		\addplot[
		only marks,    
		color=black,
		mark=*,
		mark size=1pt,
		]
		coordinates {(1,1/7)(2,4/7)(3,9/7)(4,16/7)(1/2,1/28)(3/2,9/28)(-1,1/7)(-2,4/7)(-3,9/7)(-4,16/7)(-1/2,1/28)(-3/2,9/28)};
		\end{axis}
		\end{tikzpicture}
	\end{center}	
\begin{rem}
The computation of the dimension of the fibre of $\tilde{\va}:\Q[p_1,p_2]\ra \overline{B}/\overline{I}_H$ does not determine the dimension of $R^*_{h,0}(M)/(I_H\cap R^*_{h,0}(M))$ but merely provides an upper bound. We suspect that this upper bound is sharp and have verified that it is in some cases, but the computation takes significantly more time.
\end{rem}

Based on these limited computations and the connection to lower bounds on the Krull dimension of tautological rings, we make the following optimistic conjecture.
\begin{conj}\label{Aconjecture}
	Let $M$ be a closed, oriented, smooth manifold that is positively rationally elliptic and that satisfies the Halperin conjecture. If $\hat{A}(M)=0$ then $\Kdim R^*_{h,0}(M)/(I_H\cap R^*_{h,0}(M))>0$.
\end{conj}

\bibliographystyle{alpha}
{\let\clearpage\relax \bibliography{../../Bibliography/central-bib}}

\bigskip
\textit{Email address}: nils.prigge@math.su.se
	
\end{document}